\documentclass[10pt]{siamltex}

\usepackage{graphicx}
\usepackage{amsmath}
\usepackage{amssymb}
\usepackage{color}
\usepackage{url}
\newcommand{\e}{{\mathrm{e}}}
\newcommand{\ii}{{\mathrm{i}}}

\newcommand{\R}{{\mathbb{R}}}

\newcommand{\deltab}{\boldsymbol\delta}
\newcommand{\sigmab}{\boldsymbol\sigma}
\newcommand{\bm}[1]{\mathbf{#1}}

\DeclareMathOperator{\Real}{Re}
\DeclareMathOperator{\Imag}{Im}

\numberwithin{equation}{section}

\newtheorem{remark}{Remark}[section]

\begin{document}

\title{Regularized solution of a nonlinear problem in electromagnetic
sounding}
\author{Gian Piero Deidda\thanks{Dipartimento di Ingegneria Civile, Ambientale
e Architettura, Universit\`a di Cagliari, Piazza d'Armi 1, 09123 Cagliari,
Italy. E-mail: \texttt{gpdeidda@unica.it}.} \and
Caterina Fenu\thanks{Dipartimento di Matematica e Informatica, Universit\`a di
Cagliari, viale Merello 92, 09123 Cagliari, Italy. E-mail:
\texttt{kate.fenu@unica.it}, \texttt{rodriguez@unica.it}.} \and
Giuseppe Rodriguez\footnotemark[2]}


\maketitle





\section{Introduction}\label{sec:intro}

Electromagnetic induction measurements are often used for non-destructive
investigation of certain soil properties, which are affected by
the electromagnetic features of the subsurface layers, e.g.,
the electrical conductivity and the magnetic permeability.
Knowing such parameters allows one to identify inhomogeneities in the ground,
and to ascertain the presence and the spatial position of particular conductive
substances, such as metals, liquid pollutants, or saline water.
This leads to important applications in 
Geophysics \cite{callegary2007,fraser2007,martinelli2010,vanderkruk2000}, 
Hydrology, \cite{lesch1995,paine2003}, 
Agriculture \cite{corwin2005,gebbers2007,yao2010}, 
etc.

A ground conductivity meter (GCM) is a rather common device for electromagnetic
sounding, initially introduced by the Geonics company.
It is composed by two coils (a transmitter and a receiver) placed at the
extrema of a bar.
An alternating current in the transmitter coil produces a primary magnetic
field $H_P$, which induces small currents in the ground.
These currents produce a secondary magnetic field $H_S$, which is sensed by the
receiver coil.
A GCM has two operating positions, which produce different measures,
corresponding to the orientation (either vertical or horizontal) of the
electric dipole generated by the transmitter coil; see Figure~\ref{fig:em38}.
The instrument is often coupled to a GPS, so that it is possible to associate
to each measurement the geographical position where it was taken.
Its success is due to ease of use and a relatively low price.

\begin{figure}[htb]
\begin{center}
\centerline{\includegraphics[scale=0.35]{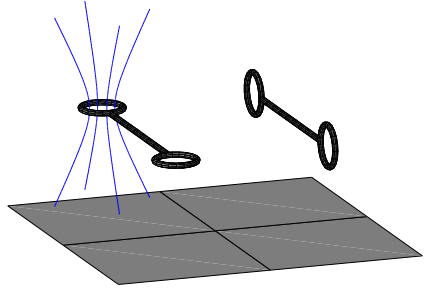}}
\end{center}
\caption{Schematic representation of a ground conductivity meter (GCM).}
\label{fig:em38}
\end{figure}

Let us assume that the instrument is placed at ground level in vertical
orientation, the soil has uniform magnetic permeability
$\mu_0=4\pi10^{-7}\,\textrm{H/m}$ (the permeability of free space) and uniform
electrical conductivity $\sigma$.
Moreover, let the \emph{induction number} be small
\begin{equation}
B = \frac{r}{\delta} = r \sqrt{\frac{\mu_0\omega\sigma}{2}} \ll 1,
\label{indnum}
\end{equation}
where
$\delta$ is the \emph{skin depth} (the depth at which the principal field $H_P$
has been attenuated by a factor $\e^{-1}$), $r$ is the inter-coil distance,
$\omega=2\pi f$, and $f$ is operating frequency of the device. 
In the case of the Geonics EM38 device, $r=1\,\textrm{m}$, 
$f=14.6\,\textrm{kHz}$, and $\delta=10\sim 50\,\textrm{m}$.
A GCM measures the apparent conductivity
\begin{equation}
m = \frac{4}{\mu_0\omega r^2}\Imag\left(\frac{(H_S)_d}{(H_P)_d}\right),
\label{appcond}
\end{equation}
which coincides with $\sigma$ under the above restrictive assumptions,
where $(H_P)_d$ and $(H_S)_d$ are the components along the dipole axis of the
primary and secondary magnetic field, respectively.

In real applications the assumption of uniform soil conductivity is not
realistic.  On the contrary, it is particularly interesting to investigate non
homogeneous soils, where the electrical conductivity $\sigma$ is not constant
and the magnetic permeability $\mu$ may be very different from $\mu_0$ for the
presence of ferromagnetic materials.

Apparent conductivity gives no information on the depth localization
of inhomogeneities.
To recover the distribution of conductivity with respect to depth by data
inversion, multiple measures are needed.
Different measures can be generated by varying some of the parameters which
influence the response of the device.
As suggested in \cite{borch97}, we assume to place the instrument at different
heights over the ground and to repeat the induction measurement with both the
possible orientations.

In 1980, McNeill \cite{mcneill80} described a linear model, based on the
response curves in the vertical and horizontal positions of the device, which
relates the apparent conductivity to the height over the ground.
If $m^V(h)$ and $m^H(h)$ are the apparent conductivity measured by the GCM at
height $h$, in the vertical and horizontal orientation, respectively, then
$$
\begin{aligned}
m^V(h) &= \int_0^\infty \phi^V(h+z) \sigma(z) \,dz, \\
m^H(h) &= \int_0^\infty \phi^H(h+z) \sigma(z) \,dz,
\end{aligned}
$$
where $z$ is the ratio between the depth and the inter-coil distance $r$,
$\sigma(z)$ is the conductivity at $z$, and
$$
\phi^V(z) = \frac{4z}{(4z^2+1)^{3/2}}, \qquad
\phi^H(z) = 2 - \frac{4z}{(4z^2+1)^{1/2}}.
$$

The linear model is valid for uniform magnetic permeability $\mu_0$, small
\emph{induction number} $B$, and moderate conductivity ($\sigma\lesssim
100\,\textrm{mS/m}$).
This model is not accurate when the conductivity of some subsurface layers is
large. In this case a nonlinear model is available \cite{hendr02,wait82}, which
will be described in the next section.

The two models are analyzed in \cite{borch97,hendr02}.
One of the conclusions is that, even if the nonlinear model produces better
results when the electrical conductivity is large, ``the linear model is
preferred for all conductivities since it needs considerably less computer
resources''.
The same authors made available two Matlab packages for inversion, based on the
linear and the nonlinear models, respectively; see \cite{borch97,hendr02}.
An algorithm for the solution of the linear model based on Tikhonov
regularization has been analyzed in \cite{deidda03}.

In this paper we propose a regularized inversion procedure for the nonlinear
model, based on the coupling of the damped Gauss--Newton method with truncated
singular value decomposition (TSVD). We give an explicit representation of the
Jacobian of the nonlinear function defining the model, and show that the 
computational load required by the algorithm is not large, and allows real-time
processing. 
For this reason we think that our approach is competitive with the existing
ones, and can be effectively used in the presence of highly conductive
materials.

The plan of the paper is the following: in Section~\ref{sec:nonlin} we describe
a nonlinear model which connects the real conductivity of the soil layers to
the apparent conductivity, and in Section~\ref{sec:jacob} we compute the
Jacobian matrix of the model.
The inversion algorithm is introduced in Section~\ref{sec:invalgo}, while
Section~\ref{sec:regul} describes the regularization procedure adopted in the
inversion algorithm.
Finally, Section~\ref{sec:numex} reports the result of numerical experiments
performed both on synthetic and real data.

\section{The nonlinear model}\label{sec:nonlin}

A nonlinear model which relates the electromagnetic features of the soil to the
height of measurement is described in \cite{wait82}, and it is further analyzed
and adapted to the case of a GCM in \cite{hendr02}.
The model is derived from Maxwell's equations, keeping into account the
cylindrical symmetry of the problem, due to the fact that the magnetic field
sensed by the receiver coil is independent of the rotation of the instrument
around the vertical axis.
In the following, $\lambda$ is a variable of integration which has no
particular physical meaning. It can be interpreted as the ratio between a
length and the \emph{skin depth} $\delta$.

Following \cite[Chapter III]{wait82}, we assume that the soil has a layered
structure with $n$ layers, each of thickness $d_i$, $i=1,\dots,n$. The bottom
layer $d_n$ is assumed to be of infinite width.
Let $\sigma_k$ and $\mu_k$ be the electrical conductivity and
the magnetic permeability of the $k$-th layer, respectively, and let
$u_k(\lambda) = \sqrt{\lambda^2 + \ii\sigma_k\mu_k\omega}$,
where $\ii=\sqrt{-1}$ is the imaginary unit.  
Then, the characteristic admittance of the $k$-th layer is given by
\begin{equation}\label{charadm}
N_k(\lambda) = \frac{u_k(\lambda)}{\ii\mu_k\omega}.
\end{equation}
The surface admittance at the top of the $k$-th layer is denoted by
$Y_k(\lambda)$ and verifies the following recursion 
\begin{equation}\label{surfadm}
Y_k(\lambda) = N_k(\lambda)\frac{Y_{k+1}(\lambda)+N_k(\lambda)
\tanh(d_k u_k(\lambda))}{N_k(\lambda) + Y_{k+1}(\lambda)
\tanh(d_k u_k(\lambda))},
\quad k=n-1,\ldots,1,
\end{equation}
where $d_k$ is the width of the $k$th layer.
The recursion is initialized setting $Y_n(\lambda)=N_n(\lambda)$ at the lowest
layer. Numerically, this is equivalent to start the recursion at $k=n$ with
$Y_{n+1}(\lambda)=0$.

Now let,
\begin{equation}
R_0(\lambda) = \frac{N_0(\lambda) - Y_1(\lambda)}{N_0(\lambda) + Y_1(\lambda)},
\label{reflfact}
\end{equation}
and
\begin{equation}
\begin{aligned}
T_0(h) &= -\delta^3 \int_{0}^{\infty}
\lambda^2 e^{-2h\lambda} R_0(\lambda) 
J_0(r\lambda) \,d\lambda, \\
T_2(h) &= -\delta^2 \int_{0}^{\infty}
\lambda e^{-2h\lambda} R_0(\lambda) 
J_1(r\lambda) \,d\lambda, \\
\end{aligned}
\label{T0T2}
\end{equation}
where $J_0(\lambda)$ and $J_1(\lambda)$ are Bessel functions of the first kind
of order 0 and 1, respectively, and $r$ is the inter-coil distance.
We prefer to express the integrals \eqref{T0T2} in the variable $\lambda$,
instead than $g=\delta\lambda$ as in \cite{wait82}.
The results obtained by Wait in \cite[page 113]{wait82}, adapted to the
geometry of a GCM, give the components of the magnetic field along the dipole
axis
$$
\begin{aligned}
(H_P)_z &=-\frac{C}{r^3}, &\qquad (H_S)_z &=-\frac{C}{\delta^3}T_0(h),
&\qquad & \text{(vertical dipole)}, \\
(H_P)_y &=-\frac{C}{r^3}, &\qquad (H_S)_y &=-\frac{C}{r\delta^2}T_2(h),
&\qquad & \text{(horizontal dipole)},
\end{aligned}
$$
where $C$ is a constant; in the case of a horizontal dipole, we assume its axis
to be $y$-directed.
Substituting in \eqref{appcond}, we obtain the predicted values of the apparent
conductivity measurement $m^V(h)$ (vertical orientation of coils) and $m^H(h)$
(horizontal orientation of coils) at height $h$ above the ground 
$$
\begin{aligned}
m^V(h) &= \frac{4}{\mu_0\omega r^2} \Imag(B^3T_0(h)), \\
m^H(h) &= \frac{4}{\mu_0\omega r^2} \Imag(B^2T_2(h)),
\end{aligned}
$$
where $B$ is the induction number \eqref{indnum}.

Simplifying formulae, we find
\begin{equation}
\begin{aligned}
m^V(h) &= \frac{4r}{\mu_0\omega} \mathcal{H}_0\left[
-\lambda e^{-2h\lambda} \Imag(R_0(\lambda)) \right](r) \\
m^H(h) &= \frac{4}{\mu_0\omega} \mathcal{H}_1\left[
-e^{-2h\lambda} \Imag(R_0(\lambda)) \right](r).
\end{aligned}
\label{mvmh}
\end{equation}
Here we denote by
\begin{equation}
\mathcal{H}_\nu[f](r) = \int_{0}^{\infty} f(\lambda) J_\nu(r\lambda) \lambda
\,d\lambda 
\label{hankel}
\end{equation}
the Hankel transform of order $\nu$ of the function $f(\lambda)$.
In our numerical experiments we approximate $\mathcal{H}_\nu[f](r)$ by the
quadrature formula described in \cite{anderson1979}, using the nodes and
weights adopted in \cite{hendr02}.

\begin{remark}
\rm
The above relations \eqref{mvmh} show that the apparent conductivity predicted
by the model is independent of the \emph{skin depth} $\delta$ and the
\emph{induction number} $B$.
To our knowledge, this is the first time that this is noted.
\end{remark}
\medskip

The model just described depends upon a number of parameters which influence
the value of the apparent conductivity. In particular, it is affected by the
instrument orientation (horizontal/vertical), its height $h$ over the ground,
the inter-coil distance $r$, and the angular frequency $\omega$.

The problem of data inversion is very important in Geophysics, when one 
is interested in depth localization of inhomogeneities of the soil.
To this purpose, multiple measures are needed to recover the distribution of
conductivity with respect to depth. In order to obtain such measures, we use
the two admissible orientations and assume to record apparent conductivity at
height $h_i$, $i=1,\ldots,m$. This generates $2m$ data values.

In our analysis, we let the magnetic permeability take the same value $\mu_0$
in the $n$ layers. This assumption is approximately met if the ground does not
contain ferromagnetic materials.
Then, we can consider the apparent conductivity as a function of the value of
the conductivity $\sigma_k$ in each layer and of the height $h$, and we write
$m^V(\sigmab,h)$ and $m^H(\sigmab,h)$, where
$\sigmab=(\sigma_1,\ldots,\sigma_n)^T$, instead than $m^V(h)$ and $m^H(h)$.

Now, let $b^V_i$ and $b^H_i$ be the data recorded by the GCM at height $h_i$ in
the vertical and horizontal orientation, respectively, and let us
denote by $r_i(\sigmab)$ the error in the model prediction for the $i$th
observation
\begin{equation}
r_i(\sigmab) = \begin{cases}
b^V_i - m^V(\sigmab,h_i), \qquad & i=1,\dots,m, \\
b^H_{m-i} - m^H(\sigmab,h_{m-i}), \qquad & i=m+1,\dots,2m.
\end{cases}
\label{risigma}
\end{equation}
Setting $\bm{b}^V=(b^V_1,\ldots,b^V_m)^T$,
$\bm{m}^V(\sigmab)=(m^V(\sigmab,h_1),\ldots,m^V(\sigmab,h_m))^T$, and defining
$\bm{b}^H$ and $\bm{m}^H(\sigmab)$ similarly, we can write the measured data
vector and the model predictions vector as
\begin{equation}
\bm{b} = \begin{bmatrix} \bm{b}^V \\ \bm{b}^H \end{bmatrix}, \quad
\bm{m}(\sigmab) = \begin{bmatrix} \bm{m}^V(\sigmab,\bm{h}) \\
\bm{m}^H(\sigmab,\bm{h}) \end{bmatrix},
\label{bmsigma}
\end{equation}
and the residual vector as
\begin{equation}
\bm{r}(\sigmab) = \bm{b} - \bm{m}(\sigmab).
\label{rsigma}
\end{equation}

To estimate the computational complexity needed to evaluate $\bm{r}(\sigmab)$
we assume that the complex arithmetic operations are implemented according to
the classical definitions, i.e., that 2 floating point operations
(\emph{flops}) are required for each complex sum, 6 for each product and 11 for
each division.  The count of other functions (exponential, square roots, etc.)
is given separately because it is not clear how many \emph{flops} they require.
If $n$ is the number of layers, $2m$ the number of data values, and $q$ the
nodes in the quadrature formula used to approximate \eqref{hankel}, we obtain a
complexity $O((45n+8m)q)$ \emph{flops} plus $2nq$ evaluations of functions with
a complex argument, and $mq$ with a real argument.

\section{Computing the Jacobian matrix}\label{sec:jacob}

As we will see in the next section, being able to compute or to approximate
the Jacobian matrix $J(\sigmab)$ of the vector function \eqref{rsigma} is
crucial for the implementation of an effective inversion algorithms and to have
information about its speed of convergence and conditioning.

The approach used in \cite{hendr02} is to resort to a finite difference
approximation
\begin{equation}
\frac{\partial r_i(\sigmab)}{\partial \sigma_j} =
\frac{r_i(\sigmab+\deltab_j)-r_i(\sigmab)}{\delta},
\quad i=1,\ldots,2m,\ j=1,\ldots,n,
\label{findiff}
\end{equation}
where $\deltab_j=\delta\,\bm{e}_j=(0,\ldots,0,\delta,0,\ldots,0)^T$ and
$\delta$ is a fixed constant.

In this section we describe the explicit expression of the Jacobian matrix.
We will show that the complexity of this computation is smaller than that
required by the finite difference approximation \eqref{findiff}.
In the following lemma we omit for clarity the variable $\lambda$.

\begin{lemma}\label{lem:derY}
The derivatives  $Y'_{kj} = \frac{\partial Y_k}{\partial \sigma_j}$,
$k,j=1,\ldots,n$, of the surface admittances \eqref{surfadm} can be obtained
starting from 
\begin{equation}
Y'_{nn} = \frac{1}{2u_n}, \qquad  
Y'_{nj} = 0, \quad j=1,\ldots,n-1,
\label{recinit}
\end{equation}
and proceeding recursively for $k=n-1,n-2,\dots,1$ by
\begin{equation}
\begin{aligned}
Y'_{kj} &= N_k^2 b_k Y'_{k+1,j}, \qquad j=n,n-1,\ldots,k+1, \\
Y'_{kk} &= \frac{a_k}{2u_k} 
	+ \frac{b_k}{2} \left[ N_k^2 d_k
	- Y_{k+1}\left(d_k Y_{k+1}
	+ \frac{1}{\ii\mu_k\omega}\right)\right], \\
Y'_{kj} &= 0, \qquad j=k-1,k-2,\ldots,1, \\
\end{aligned}
\label{yprime}
\end{equation}
where
\begin{equation}
a_k = \frac{Y_{k+1}+N_k \tanh(d_k u_k)}{N_k + Y_{k+1} \tanh(d_k u_k)}, \quad
b_k = \frac{1}{[N_k + Y_{k+1} \tanh(d_k u_k)]^2 \cosh^2(d_k u_k)}.
\label{akbk}
\end{equation}
\end{lemma}

\begin{proof}
From \eqref{charadm} we obtain
\begin{equation}
\frac{\partial u_k}{\partial \sigma_j}
	= \frac{\partial}{\partial \sigma_j}
	\sqrt{\lambda^2 + \ii\sigma_k\mu_k\omega}
	= \frac{1}{2N_k} \delta_{kj}, \qquad
\frac{\partial N_k}{\partial \sigma_j}
	= \frac{\partial}{\partial \sigma_j}
	\frac{u_k}{\ii\mu_k\omega}
	= \frac{1}{2u_k}  \delta_{kj},
\label{duknk}
\end{equation}
where $\delta_{kj}$ is the Kronecker delta, that is $1$ if $k=j$ and $0$
otherwise.
The recursion initialization \eqref{recinit} follows from $Y_n=N_n$; see
Section \ref{sec:nonlin}.
We have 
\begin{multline*}
Y'_{kj} = \frac{\partial N_k}{\partial \sigma_j} a_k
+ N_k \cdot \frac{\frac{\partial Y_{k+1}}{\partial \sigma_j} +
\frac{\partial N_k}{\partial \sigma_j}\tanh(d_k u_k) + N_k
\frac{\partial\tanh(d_k u_k)}{\partial \sigma_j}}{N_k + Y_{k+1}
\tanh(d_k u_k)} \\
- N_k a_k \cdot \frac{\frac{\partial N_k}{\partial \sigma_j} + \frac{\partial
Y_{k+1}}{\partial \sigma_j}\tanh(d_k u_k) + Y_{k+1} 
\frac{\partial\tanh(d_k u_k)}{\partial \sigma_j}}{N_k + Y_{k+1} \tanh(d_k u_k)},
\end{multline*}
with $a_k$ defined as in \eqref{akbk}.
If $j\neq k$, then $\frac{\partial N_k}{\partial \sigma_j} = 
\frac{\partial u_k}{\partial \sigma_j} = 0$ and we obtain
$$ 
Y'_{kj} = N_k^2 \frac{\frac{\partial Y_{k+1}}{\partial \sigma_j} 
\left(1 - \tanh^2(d_k u_k)\right) }{[N_k + Y_{k+1} \tanh(d_k u_k)]^2}
= N_k^2 b_k Y'_{k+1,j}.
$$
The last formula, with $b_k$ given by \eqref{akbk}, avoids the cancellation in
$1 - \tanh^2(d_k u_k)$.

If $j = k$, after some straightforward simplifications, we get
\begin{multline*}
Y'_{kk} = \frac{\partial N_k}{\partial \sigma_k}a_k + 
\frac{N_k}{N_k + Y_{k+1} \tanh(d_k u_k) }
\biggl[ Y'_{k+1,k} (1 - a_k\tanh(d_k u_k)) \biggr. \\
\biggl. + \frac{\partial N_k}{\partial \sigma_k} (\tanh(d_k u_k) - a_k)
+ \frac{d_k}{2} \left( 1-a_k\frac{Y_{k+1}}{N_k} \right) (1-\tanh^2(d_k u_k))
\biggr].
\end{multline*}
This formula, using \eqref{akbk} and \eqref{duknk}, leads to
$$
Y'_{kk} = \frac{a_k}{2u_k} 
	+ N_k b_k \left[N_k\left(Y'_{k+1,k} + \frac{d_k}{2}\right) 
	- \frac{1}{2}Y_{k+1}\left(\frac{d_k}{N_k}Y_{k+1}
	+ \frac{1}{u_k}\right)\right].
$$
The initialization \eqref{recinit} implies that $Y'_{kj}=0$ for any $j<k$.
In particular $Y'_{k+1,k}=0$, and since $N_k/u_k$ is constant one obtains the
expression of $Y'_{kk}$ given in \eqref{yprime}.
This completes the proof.
\end{proof}

\begin{remark}
\rm
The quantity $a_k$ in \eqref{akbk} appears in the right hand side of
\eqref{surfadm}, and its denominator is present also in $b_k$.
It is therefore possible to implement jointly the recursions \eqref{surfadm}
and \eqref{yprime} in order to reduce the number of floating point operations
required by the computation of the Jacobian.
We also note that since we only need the partial derivatives of $Y_1$ in the
following Theorem \ref{theorem1}, we can overwrite the values of $Y'_{k+1,j}$
with $Y'_{kj}$ at each recursion step, so that only $n$ storage locations are
needed for each $\lambda$ value, instead of $n^2$.
\end{remark}
\medskip

\begin{theorem}\label{theorem1}
The partial derivatives of the residual function \eqref{rsigma} are given by
$$
\frac{\partial r_i(\sigmab)}{\partial \sigma_j} =
\begin{cases}
\displaystyle
\frac{4r}{\mu_0\omega} \mathcal{H}_0\left[ \lambda e^{-2h_i\lambda}
\Imag\left(\frac{\partial R_0(\lambda)}{\partial\sigma_j}\right) \right](r),
\quad & i=1,\ldots,m, \\
\\
\displaystyle
\frac{4}{\mu_0\omega} \mathcal{H}_1\left[ e^{-2h_{i-n}\lambda}
\Imag\left(\frac{\partial R_0(\lambda)}{\partial\sigma_j}\right) \right](r),
\quad & i=m+1,\ldots,2m,
\end{cases}
$$
for $j=1,\ldots,n$.
Here $\mathcal{H}_\nu$ ($\nu=0,1$) denotes the Hankel transform \eqref{hankel},
$r$ is the inter-coil distance, $\frac{\partial R_0(\lambda)}{\partial
\sigma_j}$ is the $j$th component of the gradient of the function
\eqref{reflfact}
$$
\frac{\partial R_0(\lambda)}{\partial \sigma_j}
= \frac{-2\ii\mu_0\omega\lambda}{(\lambda + \ii\mu_0\omega Y_1(\lambda))^2}
\cdot \frac{\partial Y_1}{\partial \sigma_j},
$$
and the partial derivatives $\frac{\partial Y_1}{\partial \sigma_j}$ are given
by Lemma~\ref{lem:derY}.
\end{theorem}

\begin{proof}
The proof follows easily from Lemma~\ref{lem:derY} and from equations
\eqref{reflfact}, \eqref{mvmh}, and \eqref{risigma}.
\end{proof}

\begin{remark}
\rm
The numerical implementation of the above formulae needs care.
It has already been noted in the proof of Lemma \ref{lem:derY} that equations
\eqref{yprime}--\eqref{akbk} are written in order to avoid cancellations that
may introduce huge errors in the computation.
Moreover, to prevent overflow in the evaluation of the term
$$
\cosh^2(d_k u_k(\lambda)) 
= \cosh^2(d_k\sqrt{\lambda^2+\ii\sigma_k\mu_k\omega})
$$
in the denominator of $b_k$, we fix a value $\lambda_{\text{max}}$ and for
$\Real(d_k u_k(\lambda))>\lambda_{\text{max}}$ we let $b_k=b_k(\lambda)=0$.
In our numerical experiments we adopt the value $\lambda_{\text{max}}=300$.
\end{remark}
\medskip

Under the same assumptions assumed at the end of Section \ref{sec:nonlin}, we
obtain the complexity of the joint computation of the function
$\bm{r}(\sigmab)$, defined in \eqref{rsigma}, and its Jacobian, given in
Theorem \ref{theorem1}.
It amounts to $O((3n^2+8mn)q)$ \emph{flops}, $3nq$ complex functions, and $mnq$
real functions.

To approximate the Jacobian by finite differences, as in \eqref{findiff}, one
has to evaluate $n+1$ times $\bm{r}(\sigmab)$, corresponding to
$O((45n^2+8mn)q)$ \emph{flops}, $2n^2q$ complex functions, and $mnq$ real
functions.
It is immediate to observe that the computation of the Jacobian is not more
time consuming than its approximation by finite differences, and that for a
moderately large $n$ it is much faster to directly compute it, instead than
using an approximation.

In order to further reduce the computational cost, it is possible to resort to
the \emph{Broyden update} of the Jacobian, which can be interpreted as a
generalization of the secant method.
Let us denote with $J_0=J(\sigmab_0)$ the Jacobian of the function
$\bm{r}(\sigmab)$ computed in the initial point $\sigmab_0$.
Then, the Broyden update consists of applying the following recursion
\begin{equation}
J_k = J_{k-1} +
	\frac{(\bm{y}_k-J_{k-1}\bm{s}_k)\bm{s}_k^T}{\bm{s}_k^T\bm{s}_k},
\label{broyden}
\end{equation}
where $\bm{s}_{k}=\sigmab_k-\sigmab_{k-1}$ and
$\bm{y}_k = r(\sigmab_{k})-r(\sigmab_{k-1})$.
This formula makes the linearization
$$
r_k(\sigmab) = r(\sigmab_{k}) + J_k (\sigmab-\sigmab_{k})
$$
exact in $\sigmab_{k-1}$ and guarantees the least change in the Frobenius norm
$\|J_k-J_{k-1}\|_F$.
The usual approach is to apply recursion \eqref{broyden} for $1,\ldots,k_B-1$,
and to recompute the Jacobian after $k_B$ iterations, before reapplying the
update, in order to improve accuracy.
A single application of \eqref{broyden} takes $10mn+2(m+n)$ \emph{flops}, to be
added to the cost of the evaluation of $\bm{r}(\sigmab)$.
We will investigate the performance of this method in the numerical
experiments.

\section{Inversion algorithm}\label{sec:invalgo}

Let the measured data vector $\bm{b}$, the model predictions vector
$\bm{m}(\sigmab)$, and the residual vector $\bm{r}(\sigmab)$, be defined as in
\eqref{bmsigma}--\eqref{rsigma}.
The problem of data inversion, which is crucial in order to recover the
inhomogeneities of the soil, consists of computing the conductivity $\sigma_i$
of each layer ($i=1,\ldots,n$) which determine a given data set
$\bm{b}\in\R^{2m}$.
As it is customary, we use a least squares approach, by solving the nonlinear
problem
\begin{equation}\label{least}
\displaystyle \min_{\sigmab\in\R^n} f(\sigmab), \qquad
f(\sigmab) = \frac{1}{2} \|\bm{r}(\sigmab)\|^2
= \frac{1}{2}\sum_{i=1}^{2m}r_i^2(\sigmab),
\end{equation}
where $\|\cdot\|$ denotes the Euclidean norm and $r_i(\sigmab)$ is defined in
\eqref{risigma}.

The vector $\sigmab^*$ is a local minimizer of \eqref{least} if and only if it
is a stationary point, i.e., if $\bm{f}'(\sigmab^*)=0$, where
$\bm{f}'(\sigmab)$ is the gradient of the function $f$, defined by
\begin{equation}
[\bm{f}'(\sigmab)]_{j} = \frac{\partial f(\sigmab)}{\partial \sigma_j} 
= \sum_{i=1}^{2m} r_i(\sigmab)\frac{\partial r_i(\sigmab)}{\partial \sigma_j},
\qquad j=1,\ldots,n;
\label{gradf}
\end{equation}
see, e.g.,~\cite{bjo96} for a complete treatment.
We assume that $f$ is differentiable and smooth enough that the following
Taylor expansion
$$
\bm{f}'(\sigmab + \bm{s}) 
= \bm{f}'(\sigmab) + \bm{f}''(\sigmab)\bm{s} + O(\|\bm{s}\|^2) 
\simeq \bm{f}'(\sigmab) + \bm{f}''(\sigmab)\bm{s} 
$$
is valid for $\|\bm{s}\|$  sufficiently small, where 
\begin{equation}
[\bm{f}''(\sigmab)]_{jk} = \frac{\partial^2 f(\sigmab)}{\partial \sigma_j
\partial \sigma_k} = \sum_{i=1}^{2m} \left(\frac{\partial r_i(\sigmab)}{\partial \sigma_j} \frac{\partial r_i(\sigmab)}{\partial \sigma_k} +  r_i(\sigmab)\frac{\partial^2 r_i(\sigmab)}{\partial \sigma_j \partial \sigma_k}\right)
\label{hessf}
\end{equation}
is the Hessian of the function $f$.

Newton's method chooses the step $\bm{s}_\ell$ by imposing that $\sigmab^*$ is
a stazionary point, i.e., as the solution to
$$
\bm{f}''(\sigmab_\ell)\bm{s}_{\ell} = - \bm{f}'(\sigmab_\ell).
$$
The next iterate is then computed as
$\sigmab_{\ell+1}=\sigmab_\ell+\bm{s}_\ell$.
The analytic expression of the Hessian $\bm{f}''(\sigmab)$ is not always
available; whenever it is, its computation implies a large computational cost.
To overcome this problem, one possibility is to resort to the Gauss--Newton
method, which is based on the solution of a sequence of linear approximations
of $\bm{r}(\sigmab)$, rather than of $\bm{f}'(\sigmab)$.

Let $\bm{r}$ be Fr\'echet differentiable and $\sigmab_k$ denote the current
approximation, then we can write 
$$
\bm{r}(\sigmab_{k+1}) \simeq \bm{r}(\sigmab_{k}) + J(\sigmab_{k})\bm{s}_k,
$$
where $\sigmab_{k+1} = \sigmab_k + \bm{s}_k$ and $J(\sigmab)$ is the Jacobian
of $\bm{r}(\sigmab)$, defined by
$$
[J(\sigmab)]_{ij} = \frac{\partial r_i(\sigmab)}{\partial \sigma_j},
\qquad i=1,\ldots,2m, \ j=1,\ldots,n.
$$
At each step $k$, $\bm{s}_k$ is the solution of the linear least squares
problem
\begin{equation}\label{gaussnewt}
\displaystyle \min_{\bm{s}\in\R^n} \|\bm{r}(\sigmab_k) + J_k\bm{s}\|,
\end{equation}
where $J_k=J(\sigmab_k)$ or some approximation; see, e.g.,~\eqref{findiff}
and~\eqref{broyden}.

Problem \eqref{gaussnewt} is equivalent to the normal equation
\begin{equation}
J_k^T J_k \bm{s} = - J_k^T \bm{r}(\sigmab_k),
\label{normal}
\end{equation}
from which we obtain the following iterative method
\begin{equation}\label{gaussnewt2}
\sigmab_{k+1} = \sigmab_k + \bm{s}_k 
= \sigmab_k - J_k^{\dagger} \, \bm{r}(\sigmab_k),
\end{equation}
where $J_k^{\dagger}$ is the Moore--Penrose pseudoinverse of
$J_k$~\cite{bjo96}; if $2m\geq n$ and $J_k$ has full rank, then
$J_k^{\dagger}=(J_k^TJ_k)^{-1}J_k^T$.
Using this notation, the gradient \eqref{gradf} and the Hessian \eqref{hessf}
of $f(\sigmab)$ can be written as
\begin{equation}\label{newt}
\begin{aligned}
\bm{f}'(\sigmab) &= J(\sigmab)^T \bm{r}(\sigmab), \\ 
\bm{f}''(\sigmab) &= J(\sigmab)^TJ(\sigmab) 
+ \sum_{i=1}^{2m} r_i(\sigmab) H_i(\sigmab),
\end{aligned}
\end{equation}
where 
$$
[H_i(\sigmab)]_{jk} = \frac{\partial^2 r_i(\sigmab)}{\partial \sigma_j \partial \sigma_k}
$$ 
is the Hessian of the $i$th residual $r_i(\sigmab)$.
Then, the Gauss--Newton method \eqref{gaussnewt2} can be seen as a special case
of Newton's method, obtained by neglecting the term
$\sum_{i=1}^{2m} r_i(\sigmab) H_i(\sigmab)$ from \eqref{newt}.
This term is small if either each $r_i(\sigmab)$ is mildly nonlinear at
$\sigmab_k$, or the residuals $r_i(\sigmab_k)$, $i = 1,...,2m$, are small.
Since we are focused on the nonlinear case, we do not take into account the
first assumption.
We remark that in the case of a mildly nonlinear problem, a linear model is
available \cite{borch97,mcneill80}.

When the residuals $r_i(\sigmab_k)$ are small, or when the problem is
consistent ($\bm{r}(\sigmab^*)=0$), the Gauss--Newton method can be expected to
behave similarly to Newton's method. In particular, the local convergence rate
will be quadratic for both methods.
If the above conditions are not satisfied, the Gauss--Newton method may not
converge.
We remark that, while the physical problem is obviously consistent, this is not
necessarily true in our case, since we assume a layered soil, that is, we
approximate the conductivity $\sigma(z)$ by a piecewise constant function.
Furthermore, in the presence of noise in the data the problem will certainly be
inconsistent.

To ensure convergence, the damped Gauss--Newton method replaces
the approximation \eqref{gaussnewt2} by
\begin{equation}
\sigmab_{k+1} = \sigmab_k + \alpha_k\bm{s}_k,
\label{dampedGN}
\end{equation}
where $\alpha_k$ is a step length to be determined.
To choose it, we used the Armijo--Goldstein principle~\cite{ortega1970}, which
selects $\alpha_k$ as the largest number in the sequence $2^{-i}$,
$i=0,1,\dots$, for which the following inequality holds
$$
\|\bm{r}(\sigmab_k)\|^2 - \|\bm{r}(\sigmab_k 
+ \alpha_k\bm{s}_k)\|^2 \ge \frac{1}{2} \alpha_k\|J_k\bm{s}_k \|^2.
$$

The damped method allows us to include an important physical constraint in the
inversion algorithm, i.e., the positivity of the solution.
In our implementation $\alpha_k$ is the largest step size which both satisfies
the Armijo--Goldstein principle and ensures that all the solution components are
positive.

As we will show in the following section, the problem is severely
ill-conditioned, so regularization is needed.

\section{Regularization methods}\label{sec:regul}

To investigate the conditioning of problem~\eqref{least}, we studied the
behaviour of the singular values of the Jacobian matrix $J=J(\sigmab)$ of the
vector function $\bm{r}(\sigmab)$.
Let $J=U\Gamma V^T$ be the singular value decomposition (SVD)~\cite{bjo96} of
the Jacobian, where $U$ and $V$ are orthogonal matrices of size $2m$ and $n$,
respectively, $\Gamma=\diag(\gamma_1,\ldots,\gamma_p,0,\ldots,0)$ is the
diagonal matrix of the singular values, and $p$ is the rank of $J$; its
condition number is then given by $\gamma_1/\gamma_p$.

\begin{figure}[htb]
\centering
\includegraphics[width=.48\textwidth]{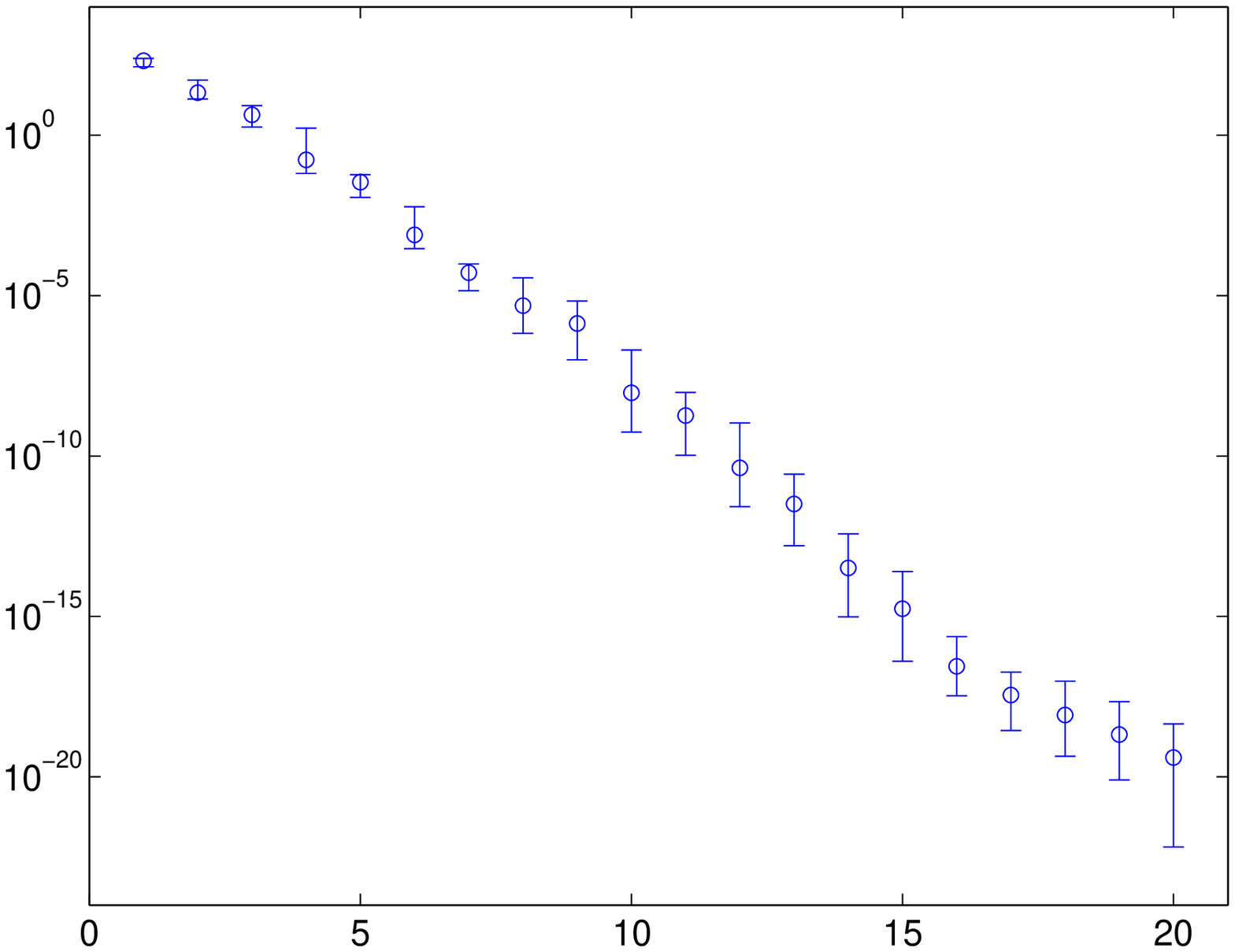}
\includegraphics[width=.48\textwidth]{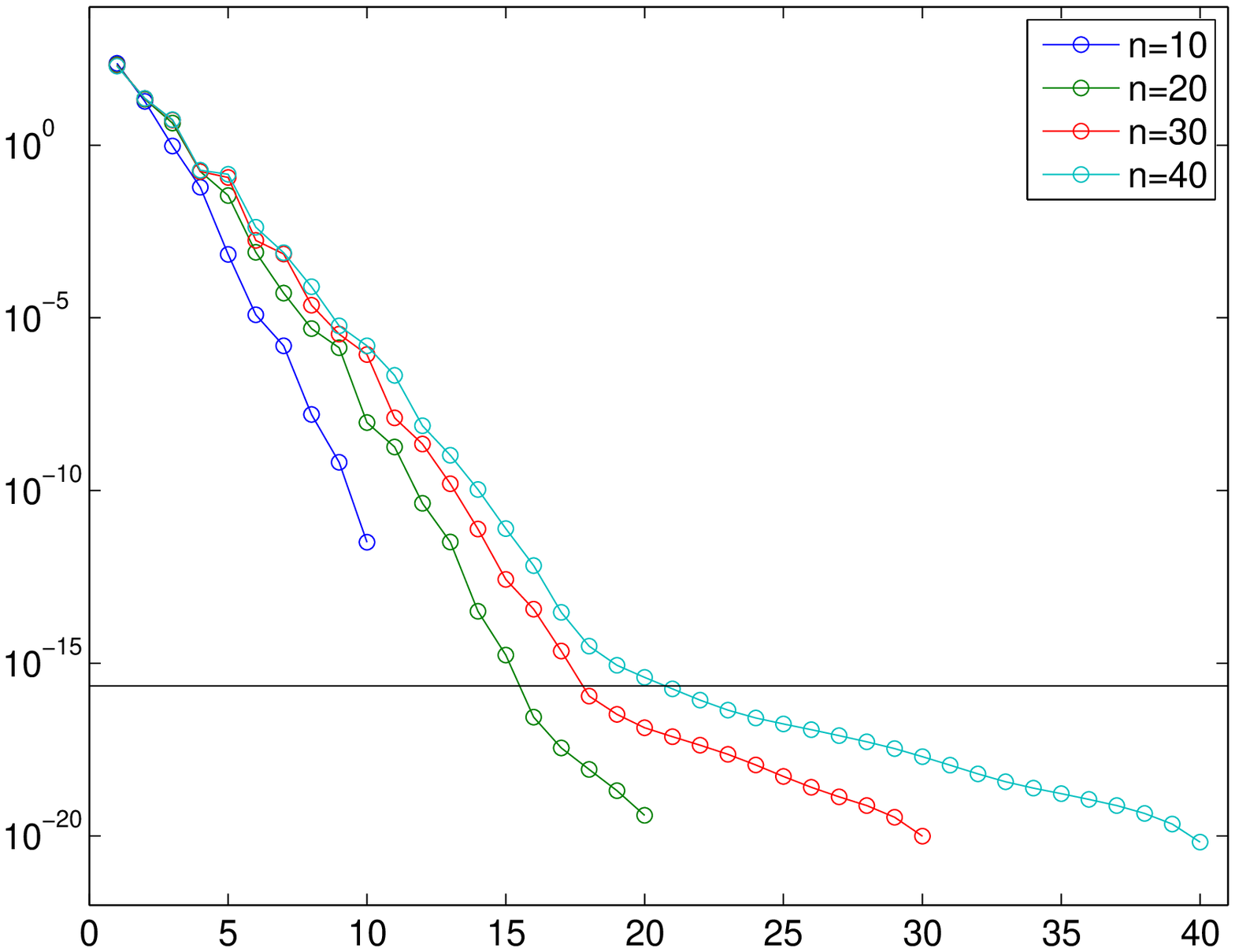}
\caption{SVD of the Jacobian matrix: left, average singular values and errors
($n=20$); right, average singular values for $n=10,20,30,40$.}
\label{fig:jacksvd}
\end{figure}

Fixed $m=10$, we generate randomly $1000$ vectors $\sigmab\in\R^{20}$, having
components in $[0,100]$. For each of them we evaluate the correponding Jacobian
$J(\sigmab)$ by the formulae proved in Theorem~\ref{theorem1} and compute its
SVD.
The left graph in Figure~\ref{fig:jacksvd} shows the average of the singular
values obtained by the above procedure and, for each of them, its minimum and
maximum value.
It is clear that deviation from the average is small, so that the
condition number of the Jacobian matrix has of the same order of magnitude in
all tests.
Consequently, the linearized problem is severely ill-conditioned independently
of the value of $\sigmab$, and we do not expect its condition number to change
much during iteration. 

The right graph in Figure~\ref{fig:jacksvd} reports the average singular values
when $n=2m=10,20,30,40$.
The figure shows that the condition number is about $10^{14}$ when $n=10$ and
increases with dimension.
The singular values appear to be exponentially decaying, so the problem is not
strictly rank-deficient.
The decay rate of singular values appears to change below machine precision
$2.2\cdot 10^{16}$, which is represented in the graph by a horinzontal line.
The exact singular vales are likely to decay with a stronger rate while the
computed ones, reported in the graph, are probably strongly perturbated by
error propagation.
A problem of this kind is generally referred to as a \emph{discrete ill-posed
problem} \cite{han98}, so regularization is needed.

A typical approach for the solution of ill-posed problems is Tikhonov
regularization. It has been applied by various author to the inversion of
geophysical data; see, e.g.,~\cite{borch97,deidda03,hendr02}.
To apply Tikhonov's method to the nonlinear problem~\eqref{least}, one has to
solve the minimization problem
\begin{equation}
\min_{\bm{\sigmab\in\R^n}} \{ \|\bm{r}(\sigmab)\|^2 + \mu^2\|L\sigmab\|^2 \}
\label{tikhonov}
\end{equation}
for a fixed value of the parameter $\mu$, where $L$ is a regularization matrix;
$L$ is often chosen as the identity matrix, or a discrete approximation of the
first or second derivative.
When the variance of the noise in the data is known, the regularization
parameter $\mu$ is usually chosen by the discrepancy principle, otherwise
various heuristic methods are used; see~\cite{han98}.
The available methods to estimate the parameter require the computation of the
regularized solution $\sigmab_\mu$ of \eqref{tikhonov} for many values of
$\mu$. This can be done, for example, by the Gauss--Newton method, leading to a
a large computational effort.

To reduce the complexity we consider an alternative regularization technique
based a low-rank approximation of the Jacobian matrix.
The best rank $\ell$ approximation ($\ell\leq p$) to the Jacobian according to
the Euclidean norm, i.e., the matrix $A_\ell$ which minimizes $\|J-A\|$ over
all the matrices of rank $\ell$, can be easily obtained by the above SVD
decomposition $J=U\Gamma V$.
This procedure allows us to replace the ill-conditioned Jacobian matrix with a
well-conditioned rank-deficient matrix $A_\ell$. 
The corresponding solution to \eqref{gaussnewt} is known as the truncated SVD
(TSVD) solution \cite{han87} and can be expressed as
\begin{equation}
\bm{s}^{(\ell)} = -A_\ell^\dagger \bm{r}
= -\sum_{i=1}^\ell \frac{\bm{u}_i^T\bm{r}}{\gamma_i} \bm{v}_i, 
\label{tsvdsol}
\end{equation}
where $\ell=1,\ldots,p$ is the regularization parameter, $\gamma_i$ are the
singular values, the singular vectors $\bm{u}_i$ and $\bm{v}_i$ are the
orthogonal columns of $U$ and $V$, respectively, and
$\bm{r}=\bm{r}(\sigmab_k)$.

To introduce a regularization matrix $L\in\R^{t\times n}$ ($t\leq n$), problem
\eqref{gaussnewt} is usually replaced by
\begin{equation}
\min_{\bm{s}\in\mathcal{S}} \|L\bm{s}\|, \qquad
\mathcal{S} = \{ \bm{s}\in\R^n ~:~ J^TJ\bm{s} = -J^T\bm{r} \},
\label{minlnorm}
\end{equation}
under the assumption $\mathcal{N}(J) \cap \mathcal{N}(L) = \{0\}$.
The generalized singular value decomposition (GSVD) \cite{paige81} of the
matrix pair $(J,L)$ is the factorization
$$
J = U \Sigma_J Z^{-1}, \qquad L = V \Sigma_L Z^{-1},
$$
where $U$ and $V$ are orthogonal matrices and $Z$ is nonsingular.
The general form of the diagonal matrices $\Sigma_J$ and $\Sigma_L$, having the
same size of $J$ and $L$, is more complicated than we need, so we analyze two
cases we are interested in.
In the case $2m\geq n=p$, the two diagonal matrices are given by
$$
\Sigma_J = \begin{bmatrix} 0 & 0 \\ C & 0 \\ 0 & I_{n-t} \end{bmatrix}, \qquad
\Sigma_L = \begin{bmatrix} S & 0 \end{bmatrix}, 
$$
where $I_{n-t}$ is the identity matrix of size $n-t$ and
$$
C = \diag(c_1,\ldots,c_t), \qquad S = \diag(s_1,\ldots,s_t),
$$
with $c_i^2+s_i^2=1$. The diagonal elements are ordered such that the
\emph{generalized singular values} $\gamma_i=c_i/s_i$ are nondecresing with
$i=1,\ldots,t$.
When $p=2m<n$, we have
$$
\Sigma_J = \begin{bmatrix} 0 & C & 0 \\ 0 & 0 & I_{n-t} \end{bmatrix}, \qquad
\Sigma_L = \begin{bmatrix} I_{n-2m} & 0 & 0 \\ 0 & S & 0 \end{bmatrix}, 
$$
where $C$ and $S$ are diagonal matrices of size $2m-n+t$. The positivity of
this number poses a constraint on the size of $L$.

The truncated GSVD (TGSVD) solution $\bm{s}_{\ell}$ to \eqref{minlnorm} is then
defined as
\begin{equation}
\bm{s}^{(\ell)} = -\sum_{i=\overline{p}-\ell+1}^{\overline{p}}
    \frac{\bm{u}_{2m-p+i}^T\bm{r}}{c_i}\, \bm{z}_{n-p+i}
    - \sum_{i=\overline{p}+1}^p (\bm{u}_{2m-p+i}^T\bm{r})\, \bm{z}_{n-p+i},
\label{tgsvdsol}
\end{equation}
where $\ell=0,1,\ldots,\overline{p}$ is the regularization parameter,
$\overline{p}=t$ if $2m\geq n$ and $\overline{p}=2m-n+t$ if $2m<n$.

Our approach to construct a regularized solution to \eqref{least} consists of
regularizing each step of the damped Gauss-Newton method \eqref{dampedGN} by
either TSVD or TGSVD.
For a fixed value of the regularization parameter $\ell$, we substitute
$\bm{s}$ in \eqref{dampedGN} by $\bm{s}^{(\ell)}$ expressed by either
\eqref{tsvdsol} or \eqref{tgsvdsol}.
We let the resulting method
\begin{equation}
\sigmab_{k+1}^{(\ell)} = \sigmab_k^{(\ell)} + \alpha_k\bm{s}_k^{(\ell)}
\label{regdampedGN}
\end{equation}
iterate until
$$
\|\sigmab_k^{(\ell)}-\sigmab_{k-1}^{(\ell)}\| < \tau \|\sigmab_k^{(\ell)}\|
\quad \text{or} \quad k > 100 \quad \text{or} \quad \alpha_k < 10^{-5},
$$
for a given tolerance $\tau$.
The constraint on $\alpha_k$ is due to its role in ensuring the positivity of
the solution. Indeed, when the solution blows up because of ill-conditioning
the damping parameter assumes very small values.
We denote the solution at convergence by $\sigmab^{(\ell)}$.
We will discuss the choice of $\ell$ in the next subsection.

\subsection{Choice of the regularization parameter}

In the previous Section we saw how to regularize the ill-conditioned problem
\eqref{least} with the aid of T(G)SVD. The choice of the regularization
parameter is crucial in order to obtain a good approximation $\sigmab^{(\ell)}$
of $\sigmab$. In this work we make use of some well-known methods to choose a
suitable index $\ell$.

In real-world applications experimental data are always affected by noise.
To model this situation, we assume that the data vector in the residual
function \eqref{rsigma}, whose norm is minimized in problem \eqref{least}, can
be expressed as $\bm{b}=\widehat{\bm{b}}+\bm{e}$, where $\widehat{\bm{b}}$
contains the exact data and $\bm{e}$ is the noise vector. This vector is
generally assumed to have normally distributed entries with mean zero and
common variance. 

If an accurate estimate of the norm of the error $\mathbf{e}$ in $\mathbf{b}$
is known, the value of $\ell$ can often be determined with the aid of the 
discrepancy principle~\cite[Section 4.3]{ehn96}. It consists of determining the
regularization parameter $\ell$ as the smallest index
$\ell=\ell_{\text{discr}}$ such that 
\begin{equation}\label{discrp}
\|\mathbf{b}-\bm{m}(\sigmab_{\ell_{\text{discr}}})\|\leq\kappa\|\mathbf{e}\|.
\end{equation}
Here $\kappa>1$ is a user-supplied constant independent of $\|\mathbf{e}\|$.
In our experiments we set $\kappa=1.5$, since it produced the best numerical
results.
The discrepancy principle typically yields a suitable truncation index when 
an accurate bound for $\|{\mathbf{e}}\|$ is available. 

We are also interested in the situation when an accurate bound for
$\|{\mathbf{e}}\|$ is not available and, therefore, the discrepancy principle
cannot be applied.
A large number of methods for determining a regularization parameter in this
situation have been introduced for linear inverse problems~\cite{han98}.
They are known as \emph{heuristic} because it is not possible to prove
convergence results for them, in the strict sense of the definition of a
regularization method; see, e.g.,~\cite[Chapter 4]{ehn96}.
Nevertheless, it has been shown by numerical experiments, that some heuristic
methods provide a good estimation of the optimal regularization parameter in
many inverse problems of applicative interest.

It is not possible, in general, to apply all the heuristic methods, which were
developed in the linear case, to a nonlinear problem.
In this paper we use the L-curve criterion~\cite{hol93}, which can be extended
quite naturally to the nonlinear case.
Let us consider the curve obtained by joining the points 
$$
\left\{\log{\|\bm{r}(\sigmab^{(\ell)})\|},\log{\|L \sigmab^{(\ell)} \|}
\right\}, \quad \ell = 1,\dots,\overline{p},
$$
where $\bm{r}(\sigmab^{(\ell)}) = \mathbf{b}-\bm{m}(\sigmab^{(\ell)})$ is the
residual error associated to the approximate solution $\sigmab^{(\ell)}$
computed by the iterative method \eqref{regdampedGN}, using \eqref{tgsvdsol} as
a regularization method.
If \eqref{tsvdsol} is used instead, it is sufficient to let $L=I$ and replace
$\overline{p}$ by $p$.

This curve exhibits a typical L-shape in many discrete ill-posed problems.
The L-curve criterion seeks to determine the regularization parameter by
detecting the index $\ell$ of the point of the curve closer to the corner of
the ``L''. This choice produces a solution for which both the norm and the
residual are fairly small.

Various method has been proposed to determine the corner of the L-curve.
In our numerical experiments we use two of them.
The first one, which we denote as the \emph{corner} method, considers a
sequence of pruned L-curves, obtained by removing an increasing number of
points, and constructs a list of candidate “vertices” produced by two different
selection algorithms. 
The corner is selected from this list by a procedure which compares the norms
and the residuals of the corresponding solutions~\cite{hjr07}.
It is currently implemented in~\cite{han07}.

The second procedure we use has been recently proposed in~\cite{rr13}, by
extending a method by T. Regi\'nska~\cite{reg96}, which detects the corner by
solving an optimization problem. We will refer to this method as the 
\emph{restricted Regi\'nska} (ResReg) method.

\section{Numerical experiments}\label{sec:numex}

To illustrate the performance of the inversion methods described in the
previous sections we present here the results of a set of numerical
experiments. Initially, we will apply our method to synthetic data sets,
generated by choosing a conductivity distribution and adding random noise to
data. Finally, we will analyze a real data set.

\begin{figure}[htb]
\centering
\includegraphics[width=.32\textwidth]{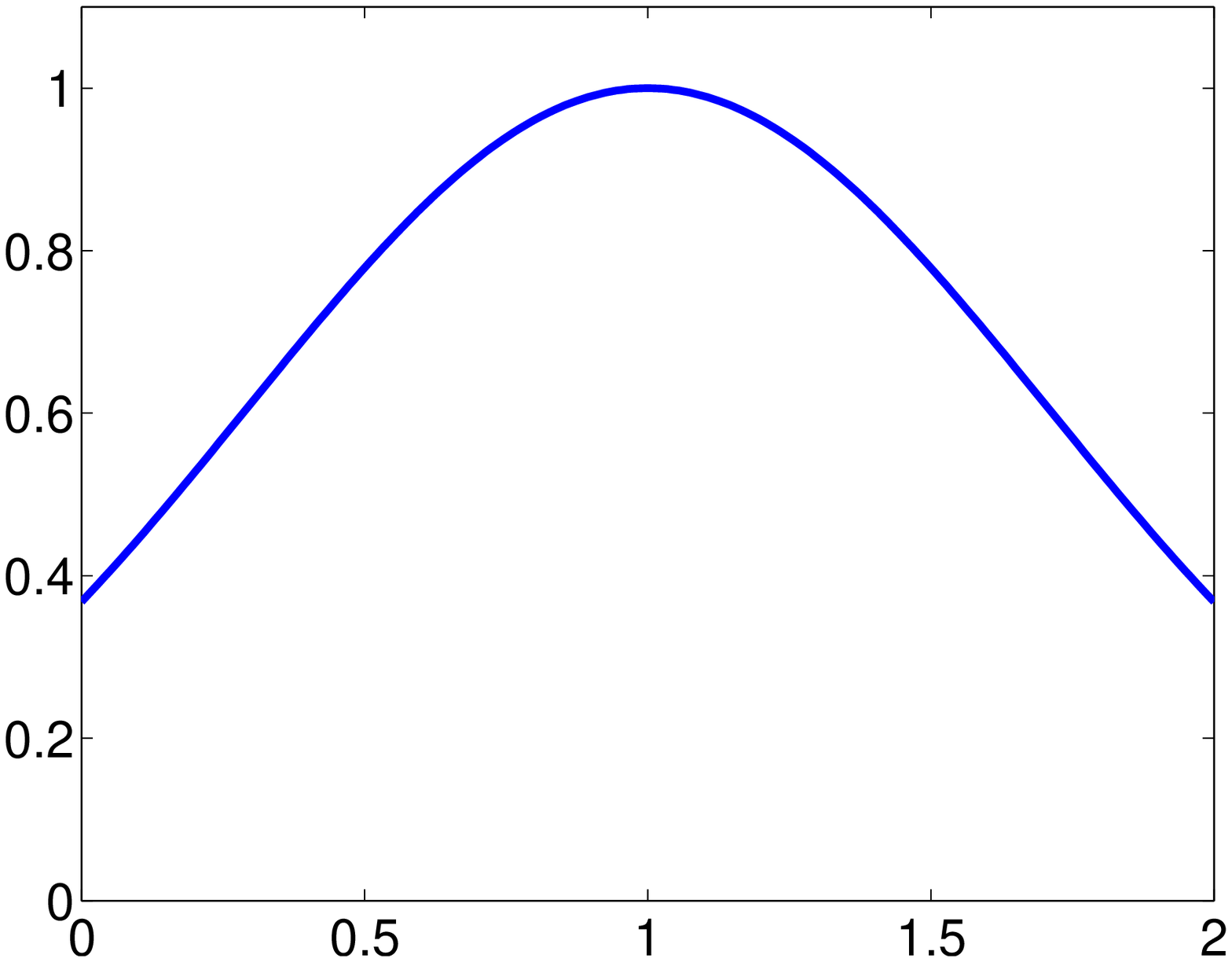}
\includegraphics[width=.32\textwidth]{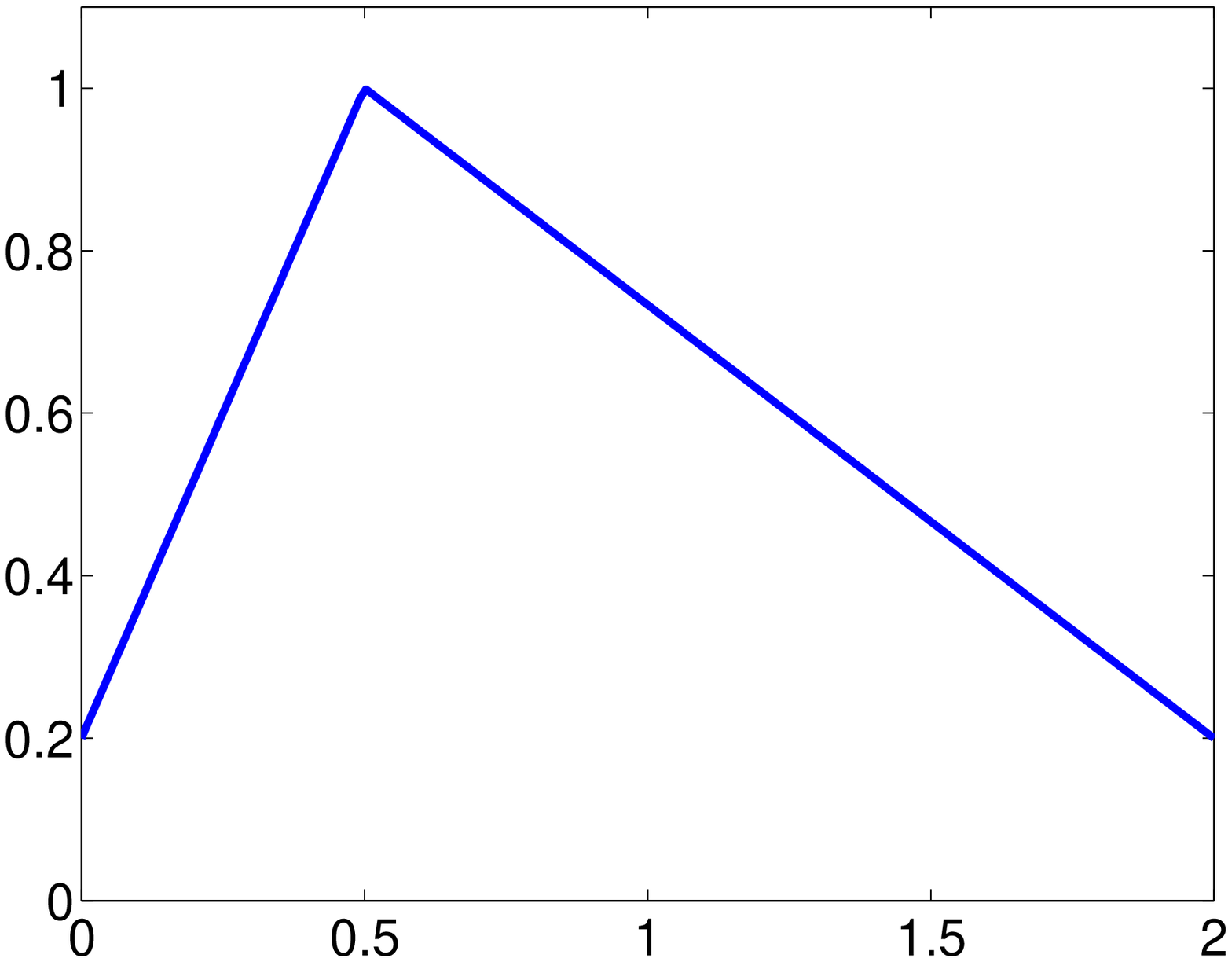}
\includegraphics[width=.32\textwidth]{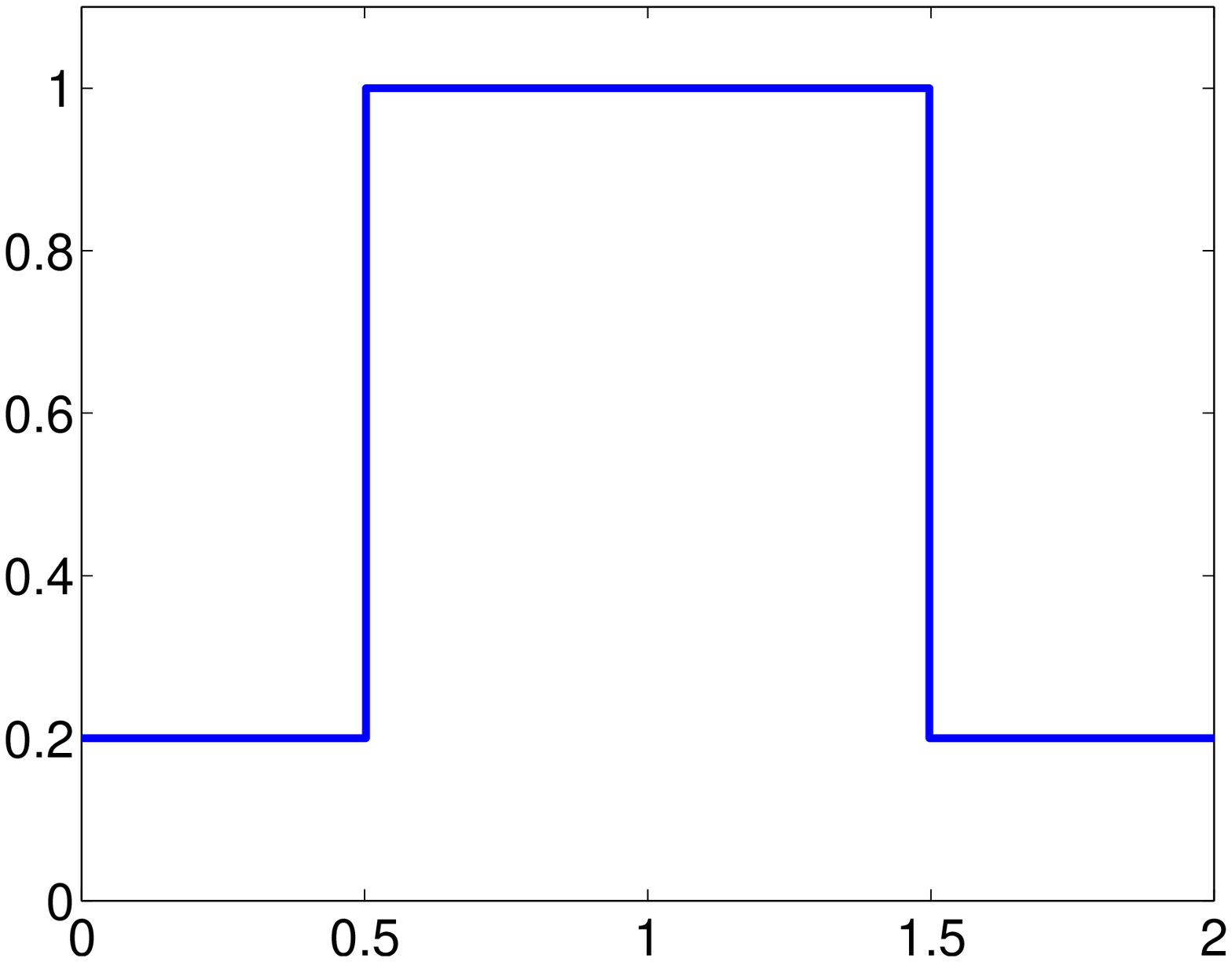}
\caption{Graphs of the conductivity distribution models $f_1$, $f_2$, and $f_3$.
The horizontal axis reports the depth in meters, the vertical axis the
electrical conductivity in Siemens/meter.}
\label{fig:funs}
\end{figure}

Figure~\ref{fig:funs} reports the three functions $f_\ell(z)$, $\ell=1,2,3$,
used in our experiments to model the distribution of conductivity, expressed in
Siemens/meter, with respect to the depth $z$, measured in meters.
The first one is differentiable ($f_1(z)=\e^{-(z-1)^2}$), the second is
piecewise linear, the third is a step function. All model functions assume the
presence of a strongly conductive material at a given depth.

For a chosen model function $f_k$ and a fixed number of layers $n$, we let the
layers thickness assume the constant value $d_k=\bar{d}=2/(n-1)$,
$k=1,\ldots,n-1$ (see Section~\ref{sec:nonlin}), so that $z_j=(j-1)\bar{d}$,
$j=1,\ldots,n$.
The choice of $\bar{d}$ is motivated by the common assumption that a GCM can
give useful information about the conductivity of the ground up to a depth of 2
meters. This fact is confirmed by our experiments.

We assign to each layer the conductivity $\sigma_j=f_k(z_j)$.
Then, we apply the nonlinear model \eqref{bmsigma} to compute the exact data
vector $\widehat{\mathbf{b}}$, letting 
$$
\widehat{b}_i = \begin{cases}
\hat{b}^V_i = m^V(\sigmab,h_i), \quad & i=1,\dots,m, \\
\hat{b}^H_{m-i} = m^H(\sigmab,h_{m-i}), \quad & i=m+1,\dots,2m.
\end{cases}
$$
We assume that the measurements are taken with the EMS in both vertical and
horizontal orientation, placed at the heights $h_i=(i-1)\bar{h}$ above the
ground, $i=1,\ldots,m$, for a chosen height step $\bar{h}$; see \eqref{risigma}.
In our experiments $\bar{h}\geq 0.1\mathrm{m}$.

To simulate experimental errors, we determine the perturbed data vector
$\mathbf{b}$ by adding a noise vector to $\widehat{\mathbf{b}}$.
Specifically, we let the vector $\mathbf{w}$ have normally 
distributed entries with mean zero and variance one, and compute
$$
\mathbf{b}=\widehat{\mathbf{b}}+\mathbf{w} \, \|\widehat{\mathbf{b}}\|
\frac{\tau}{\sqrt{2m}}.
$$
This implies that
$\|\mathbf{b}-\widehat{\mathbf{b}}\|\approx\tau\|\widehat{\mathbf{b}}\|$.
In the computed examples we use the noise levels
$\tau=10^{-3},10^{-2},10^{-1}$.
The value of $\tau$ is used in the discrepancy principle (\ref{discrp}), 
where we substitute $\tau\|\widehat{\mathbf{b}}\|$ for $\|\mathbf{e}\|$.

For each data set, we solve the least squares problem \eqref{least} by the
damped Gauss--Newton method \eqref{dampedGN}.
The damping parameter is determined by the Armijo--Goldstein principle,
modified in order to ensure the positivity of the solution.
Each step of the iterative method is regularized by either the TSVD approach
\eqref{tsvdsol}, or by TGSVD \eqref{tgsvdsol}, for a given regularization
matrix $L$. In our experiments we use both $L=D_1$ and $L=D_2$, the discrete
approximations of the first and second derivatives.
This two choices pose a constraint on the magnitude of the slope and the
curvature of the solution, respectively.
To assess the accuracy of the computations we use the relative error
\begin{equation}
e_\ell = \frac{\|\sigmab-\sigmab^{(\ell)}\|}{\|\sigmab\|},
\label{error}
\end{equation}
where $\sigmab$ denotes the exact solution of the problem and
$\sigmab^{(\ell)}$ its regularized solution with parameter $\ell$, obtained by
\eqref{regdampedGN}.
The experiments were performed using Matlab 8.1 (R2013a) on an Intel Core
i7/860 computer with 8Gb RAM, running Linux.
The software developed is available from the authors upon request.

\begin{table}
\caption{Optimal error $e_{\text{opt}}$ for $m=5,10,20$ and $n=20,40$, for the
TSVD solution ($L=I$) and the TGSVD solution with $L=D_1$ and $L=D_2$. The
Jacobian is computed as in Section~\ref{sec:jacob}.}
\label{tab:example2}
\begin{center}
\begin{tabular}{cc|cc|cc|cc}
& \multicolumn{1}{c}{} & \multicolumn{2}{c}{$L=I$} 
	& \multicolumn{2}{c}{$L=D_1$} & \multicolumn{2}{c}{$L=D_2$} \\
example & \multicolumn{1}{c}{$m$} & $n=20$ & \multicolumn{1}{c}{$n=40$} 
	& $n=20$ & \multicolumn{1}{c}{$n=40$} 
	& $n=20$ & \multicolumn{1}{c}{$n=40$} \\
\hline
      &   5 & 2.4e-01 & 2.4e-01 & 8.6e-02 & 8.0e-02 & 6.9e-02 & 7.0e-02 \\
$f_1$ &  10 & 2.2e-01 & 2.1e-01 & 5.2e-02 & 5.7e-02 & 5.2e-02 & 4.6e-02 \\
      &  20 & 2.2e-01 & 2.2e-01 & 3.9e-02 & 4.9e-02 & 3.1e-02 & 3.5e-02 \\
\hline
      &   5 & 3.1e-01 & 3.7e-01 & 7.2e-02 & 6.4e-02 & 9.7e-02 & 1.2e-01 \\
$f_2$ &  10 & 2.8e-01 & 3.5e-01 & 6.3e-02 & 6.2e-02 & 7.3e-02 & 8.2e-02 \\
      &  20 & 2.8e-01 & 3.9e-01 & 6.5e-02 & 5.9e-02 & 7.9e-02 & 7.2e-02 \\
\hline
      &   5 & 4.2e-01 & 4.6e-01 & 2.9e-01 & 2.9e-01 & 2.9e-01 & 3.0e-01 \\
$f_3$ &  10 & 3.5e-01 & 4.7e-01 & 2.7e-01 & 2.6e-01 & 2.7e-01 & 2.8e-01 \\
      &  20 & 3.3e-01 & 4.7e-01 & 2.6e-01 & 2.6e-01 & 2.7e-01 & 2.9e-01 \\
\hline
\end{tabular}
\end{center}
\end{table}

Our first experiment tries to determine the optimal experimental setting, that
is, the number of measurements to be taken and the number of underground layers
to be considered.
At the same time, we investigate the difference between the TSVD
\eqref{tsvdsol} and the TGSVD \eqref{tgsvdsol} approaches, and the effect on
the solution of the regularization matrix $L$.
For each of the three test conductivity models, we discretize the soil by 20 or
40 layers, up to the depth of 2m.
We generate synthetic measures at 5, 10, and 20 equispaced heights up to 1.9m,
and we solve the problem.
This process is repeated for each regularization matrix.
The (exact) Jacobian is computed as described in Section~\ref{sec:jacob}.
Table~\ref{tab:example2} reports the values of the relative error
$e_{\text{opt}}=\min_\ell e_\ell$, representing the best possible performance
of the method. This value is the average over 20 realizations of the noise.

\begin{figure}[htb]
\centering
\includegraphics[width=.49\textwidth]{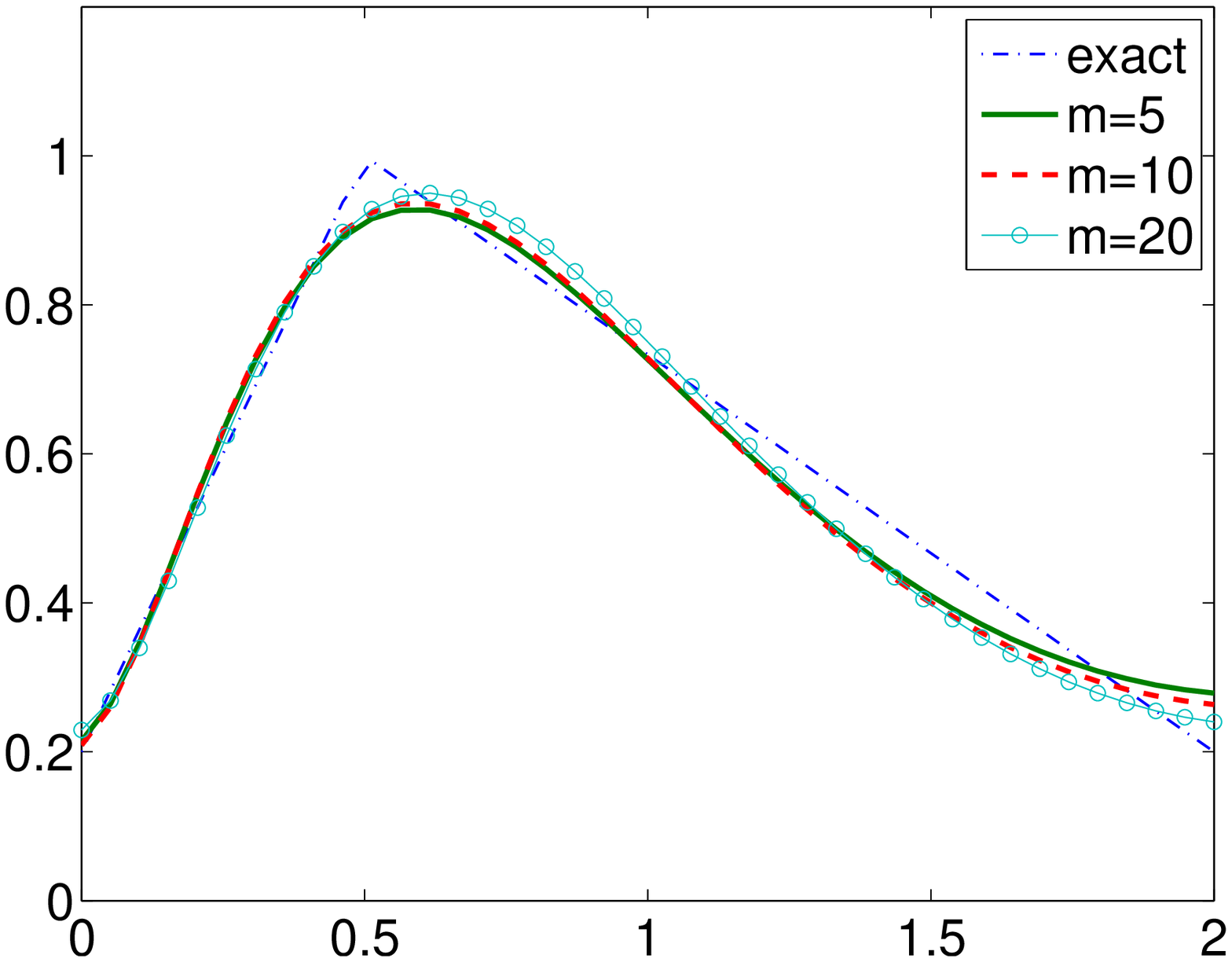}
\includegraphics[width=.49\textwidth]{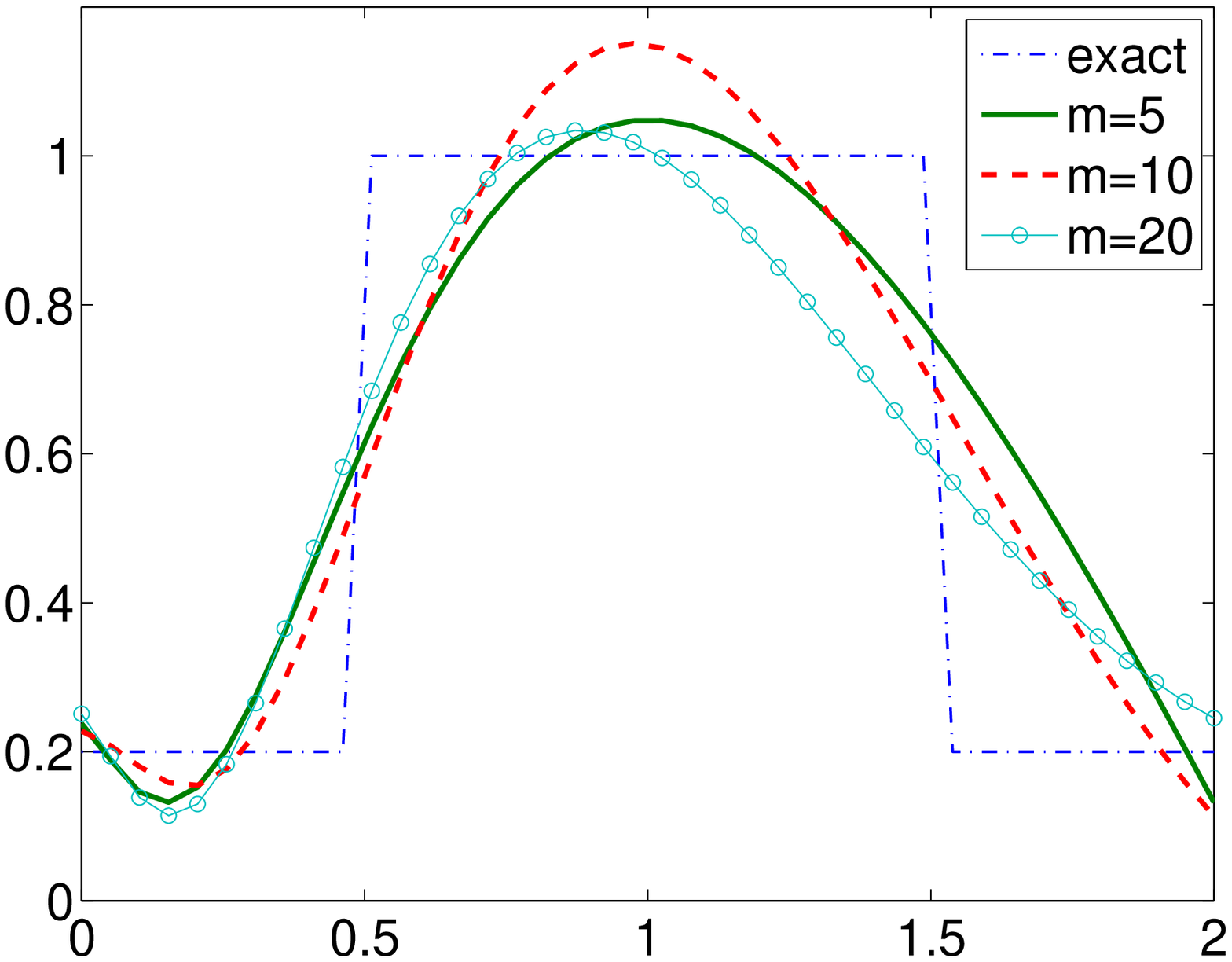}
\caption{Optimal reconstruction for the model functions $f_2$ and $f_3$. The
number of underground layers is $n=40$, the noise level is $\tau=10^{-3}$. The
solid line is the solution obtained with $m=5$, the dashed line corresponds to
$m=10$, the line with bullets to $m=20$. The exact solution is represented by a
dash-dotted line.}
\label{fig:best}
\end{figure}

It is clear that the TSVD approach is the least accurate. The TGSVD with
$L=D_2$ gives the best results for $f_1$, that is when the solution is smooth.
When the conductivity distribution is less regular, like $f_2$ and $f_3$,
the first derivative $L=D_1$ produces the more accurate approximations.
From the results, it seems convenient to use a large number of layers to
discretize the soil, that is $n=40$. This choice does not increase
significantly the computation time.
It is obviously desirable to have at disposal a large number of measurements,
however the results obtained with $m=5$ and $m=10$ are not much worse than
those computed with $m=20$, and they might be sufficient to give a rough
approximation of the depth localization of a conductive substance.
This is an important remark, as it reduces the time needed for data
acquisition.

Figure~\ref{fig:best} gives an idea of the quality of the computed
reconstructions for the model functions $f_2$ and $f_3$, with $n=40$ and 
noise level $\tau=10^{-3}$. The exact solution is compared to the
approximations corresponing to $m=5,10,20$.
The above remarks about the influence of the number of measurements $m$ is
confirmed. It is also remarkable that the position of the maximum is very well
localized.

\begin{table}
\caption{Optimal error $e_{\text{opt}}$ for $m=5,10,20$ and $n=20,40$, for 
$f_1$ $(L=D_2)$, $f_2$ $(L=D_1)$, and $f_3$ $(L=D_1)$. The results obtained
from measurements collected with the instrument in both vertical and horizontal
orientation are compared to those obtained with a single orientation.}
\label{tab:example5}
\begin{center}
\begin{tabular}{cc|cc|cc|cc}
& \multicolumn{1}{c}{} & \multicolumn{2}{c}{$f_1$, $L=D_2$} 
	& \multicolumn{2}{c}{$f_2$, $L=D_1$} 
	& \multicolumn{2}{c}{$f_3$, $L=D_1$} \\
orientation & \multicolumn{1}{c}{$m$} & $n=20$ & \multicolumn{1}{c}{$n=40$} 
	& $n=20$ & \multicolumn{1}{c}{$n=40$} 
	& $n=20$ & \multicolumn{1}{c}{$n=40$} \\
\hline
           &   5 & 6.9e-02 & 7.0e-02 & 7.2e-02 & 6.4e-02 & 2.9e-01 & 2.9e-01 \\
both       &  10 & 5.2e-02 & 4.6e-02 & 6.3e-02 & 6.2e-02 & 2.7e-01 & 2.6e-01 \\
           &  20 & 3.1e-02 & 3.5e-02 & 6.5e-02 & 5.9e-02 & 2.6e-01 & 2.6e-01 \\
\hline 
           &   5 & 1.4e-01 & 1.0e-01 & 1.8e-01 & 1.8e-01 & 3.7e-01 & 3.7e-01 \\
vertical   &  10 & 7.0e-02 & 1.2e-01 & 1.4e-01 & 1.4e-01 & 3.8e-01 & 3.5e-01 \\
           &  20 & 7.5e-02 & 7.5e-02 & 1.2e-01 & 1.1e-01 & 3.3e-01 & 3.3e-01 \\
\hline 
           &   5 & 1.3e-01 & 1.3e-01 & 2.7e-01 & 2.6e-01 & 4.4e-01 & 4.1e-01 \\
horizontal &  10 & 8.4e-02 & 6.1e-02 & 1.4e-01 & 1.2e-01 & 3.8e-01 & 4.0e-01 \\
           &  20 & 7.2e-02 & 6.7e-02 & 1.1e-01 & 8.6e-02 & 3.5e-01 & 3.4e-01 \\
\hline 
\end{tabular}
\end{center}
\end{table}

In the previous experiments we assumed that all the $2m$ entries of vector
$\bm{b}$ in \eqref{bmsigma} were available.
In Table~\ref{tab:example5} we compare these results with those obtained by
using only half of them, i.e., those corresponding to either the vertical or
horizontal orientation of the instrument.
The results with the label ``both'' in the first column are extracted from
Table~\ref{tab:example2}.
The results are slightly worse when the number of data is halved, especially
for the smooth model function, while they are almost equivalent for the step
function $f_3$.

In Section~\ref{sec:jacob} we described the computation of the Jacobian matrix
of \eqref{rsigma}, and compared it to the slower finite difference
approximation \eqref{findiff} and to the Broyden update \eqref{broyden}.
To investigate the execution time corresponding to each method, we let the
method \eqref{regdampedGN} perform 100 iterations, with $L=D_2$, for a fixed
regularization parameter ($\ell=4$). When the Jacobian is exactly computed, the
execution time is 7.18s, while the finite difference approximation requires
18.96s. The speedup factor is 2.6, which is far less than the one theoretically
expected. This is probably due to the implementation details, and to the fact
that the Matlab programming language is interpreted. We performed the same
experiment by applying the Broyden update \eqref{broyden} and recomputing the
Jacobian every $k_B$ iterations. For $k_B=5$ the execution time was 2.00s,
for $k_B=10$, 1.32s. Despite this strong speedup, the accuracy is not
substantially affected by this approach. Table~\ref{tab:example2broyden}
reports the relative error $e_{\text{opt}}$ obtained by repeating the
experiment of Table~\ref{tab:example2} using the Broyden method with $k_B=10$.
We only report the values of $e_{\text{opt}}$ for the most interesting
examples. The loss of accuracy is minimal.

\begin{table}
\caption{Optimal error $e_{\text{opt}}$ for $m=5,10,20$ and $n=20,40$, for 
$f_1$ $(L=D_2)$, $f_2$ $(L=D_1)$, and $f_3$ $(L=D_1)$. The Jacobian is computed
every 10 iterations and then updated by the Broyden method.}
\label{tab:example2broyden}
\begin{center}
\begin{tabular}{c|cc|cc|cc}
\multicolumn{1}{c}{} & \multicolumn{2}{c}{$f_1$, $L=D_2$} 
	& \multicolumn{2}{c}{$f_2$, $L=D_1$} 
	& \multicolumn{2}{c}{$f_3$, $L=D_1$} \\
\multicolumn{1}{c}{$m$} & $n=20$ & \multicolumn{1}{c}{$n=40$} 
	& $n=20$ & \multicolumn{1}{c}{$n=40$} 
	& $n=20$ & \multicolumn{1}{c}{$n=40$} \\
\hline 
 5 & 7.3e-02 & 7.6e-02 & 7.7e-02 & 7.6e-02 & 3.0e-01 & 2.9e-01 \\
10 & 5.5e-02 & 4.8e-02 & 6.9e-02 & 7.4e-02 & 2.7e-01 & 2.8e-01 \\
20 & 4.3e-02 & 4.0e-02 & 7.3e-02 & 6.9e-02 & 2.6e-01 & 2.7e-01 \\
\hline
\end{tabular}
\end{center}
\end{table}

\begin{figure}[htb]
\centering
\includegraphics[width=\textwidth]{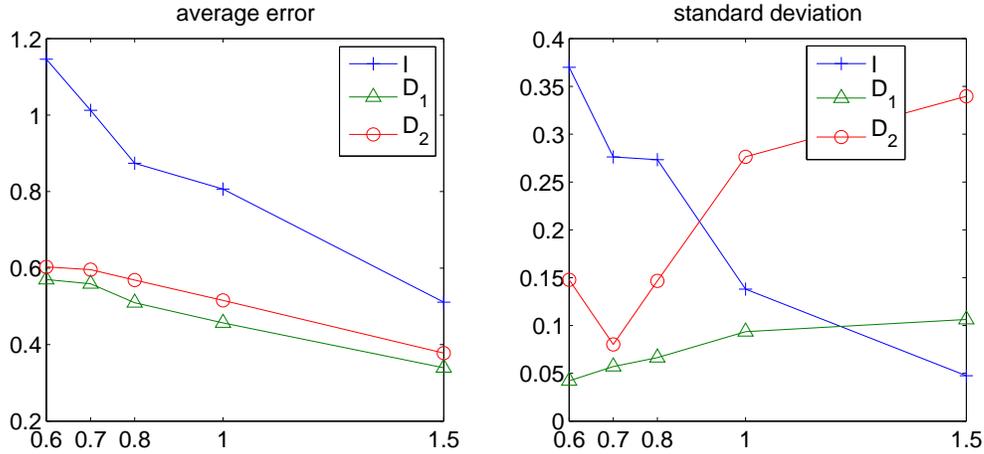}
\caption{Results for the reconstruction of test function $f_3$ with a variable
step length $\xi$, which is reported on the horizontal axis. The left graph
reports the average error $e_{\text{opt}}$, obtained with three regularization
matrices $L=I,D_1,D_2$. Each test is repeated 20 times for each noise level
$\tau=10^{-3},10^{-2},10^{-1}$. The right graph reports the corresponding
standard deviations.}
\label{fig:meanstd}
\end{figure}

Another interesting issue is understanding which is the spatial resolutions of
the inversion algorithm, that is, which is the performance of the method in the
presence of a very thin conductive layer. To this end, we consider the test
function $f_3$, and let the length $\xi$ of the step vary. Each problem is
solved for three regolarization matrices, three noise levels, and each test is
repeated 20 times for different noise realizations. The left graph of
Figure~\ref{fig:meanstd} reports the average errors for each value of $\xi$,
while the right graph displays the standard deviations. The choice $L=D_1$
appears to be the best. Indeed, not only the errors are better, but the smaller
standard deviations ensure that the method is more reliable.
Figure~\ref{fig:gradino} shows the reconstructions of $f_3$ with three
different step lengths, with $\xi=1.5,1.0,0.7$, $L=D_1$, and $\tau=10^{-2}$.
It is remarkable that the position of the maximum is well located by the
algorithm even in the presence of a very thin step.

\begin{figure}[htb]
\centering
\includegraphics[width=\textwidth]{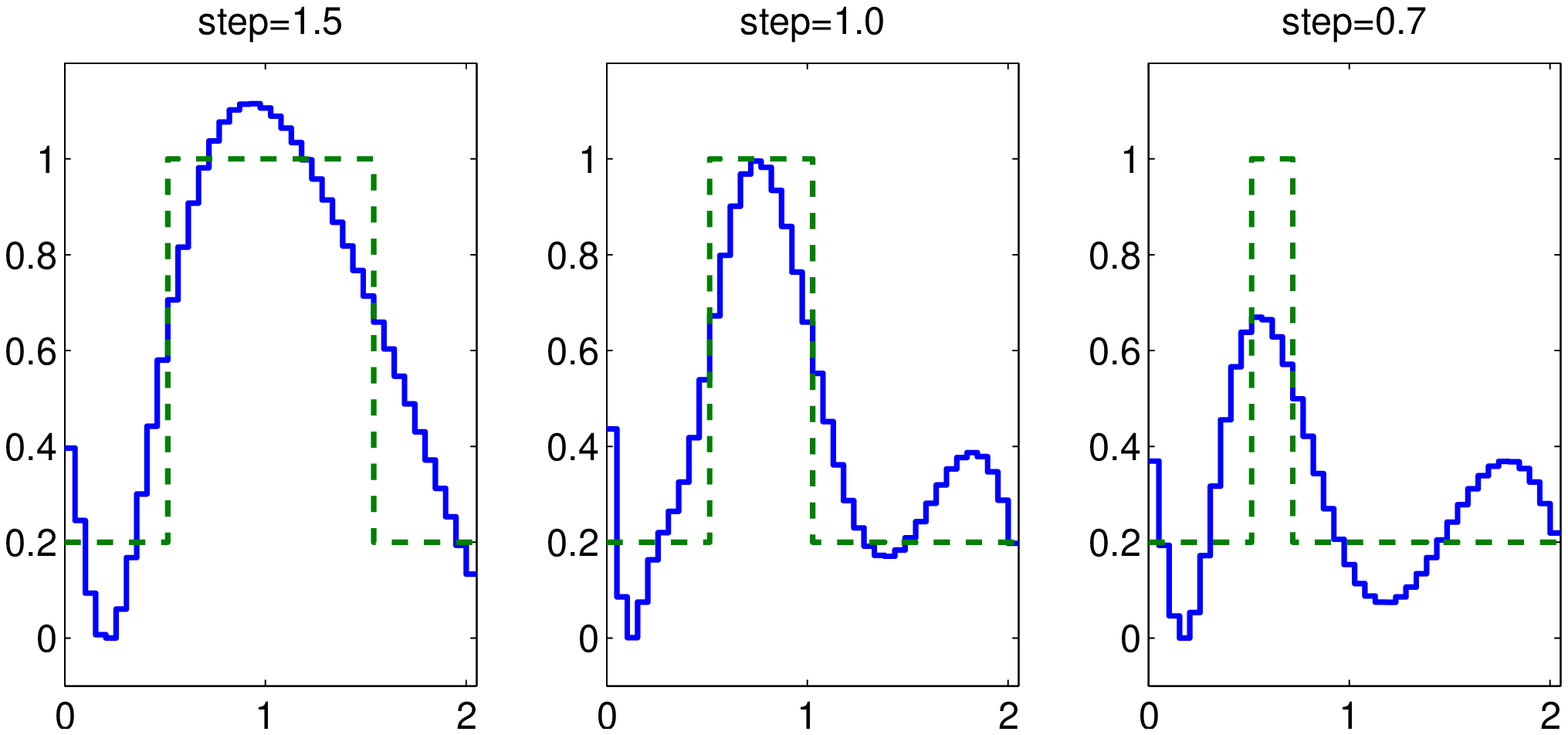}
\caption{Optimal reconstructions for the test function $f_3$, with step lengths
1.5, 1.0, and 0.7, obtained with $L=D_1$ and noise level $\tau=10^{-2}$.}
\label{fig:gradino}
\end{figure}



\end{document}